\documentclass[12pt,reqno]{amsart}

\usepackage{custom_preamble}

\begin{document}

\title{Extractors in Paley graphs: a random model}

\author{\myname}
\address{\myaddress}
\email{\myemail}

\begin{abstract}
A well-known conjecture in analytic number theory states that for every pair of sets $X,Y\subset\mathbb{Z}/p\mathbb{Z}$, each of size at least $\log ^C p$ (for some constant $C$) we have that the number of pairs $(x,y)\in X\times Y$ such that $x+y$ is a quadratic residue modulo $p$ differs from $\frac12|X||Y|$ by $o\left(|X||Y|\right)$. We address the probabilistic analogue of this question, that is for every fixed $\delta>0$, given a finite group $G$ and $A\subset G$ a random subset of density $\frac12$, we prove that with high probability for all subsets $|X|,|Y|\geq \log ^{2+\delta} |G|$, the number of pairs $(x,y)\in X\times Y$ such that $xy\in A$ differs from $\frac12|X||Y|$ by $o\left(|X||Y|\right)$.
\end{abstract}
 
\maketitle


\section{Introduction}

A folklore result in analytic number theory states that if we let $Q\subset\mathbb{Z}/p\mathbb{Z}$ be the set of quadratic residues modulo a prime $p$, then for any function $w\colon \mathbb{N}\to \mathbb{R}$ tending to infinity and any pair of sets $X,Y\subset\mathbb{Z}/p\mathbb{Z}$ of size at least $w(p)\sqrt{p}$, the number of pairs $(x,y)\in X\times Y$ such that $x+y\in Q$ differs from $\frac12|X||Y|$ by $o\left(|X||Y|\right)$. Here, the rate of convergence implied by the $o$-notation depends only on $w$.

The proof of this is relatively simple. We refer the reader to Section~\ref{S:notation} for the notation used below. First of all, after setting $1_X\ast 1_Y(x) = \mathbb{E}_{z\in\mathbb{Z}/p\mathbb{Z}}1_X(z)1_Y(x-z)$, it is easy to see that the statement one wants to prove is equivalent to requiring that
$$\left|p \sum_{x\in \mathbb{Z}/p\mathbb{Z}}1_X\ast 1_Y(x)\chi(x)\right|=o(|X||Y|),$$
where $\chi$ is the quadratic character (i.e.\ $\chi(a)=(\frac{a}{p})$). For $r\in \mathbb{Z}/p\mathbb{Z}$ and a function $f\colon \mathbb{Z}/p\mathbb{Z}\to \mathbb{R}$, we define the corresponding Fourier coefficient by $\widehat{f}(r)=\mathbb{E}_{x\in\mathbb{Z}/p\mathbb{Z}}f(x)e^{2\pi ixr/p}$. Using standard formulas from Fourier analysis, the standard estimate for Gauss sums (see e.g.\ \cite{kowalskiiwaniec}, Section 3.5), and Cauchy-Schwarz inequality, we have
\begin{align*}
 \left|p \sum_{x\in \mathbb{Z}/p\mathbb{Z}}1_X\ast 1_Y(x)\chi(x)\right| &= \left|p^2\sum_{r\in \mathbb{Z}/p\mathbb{Z}}\widehat{1_X}(r)\widehat{1_Y}(r) \bar{\widehat{\chi}}(r)\right| \leq p^{3/2} \sum_{r\in \mathbb{Z}/p\mathbb{Z}} |\widehat{1_X}(r)| |\widehat{1_Y}(r)|\\
 &\leq \sqrt{p}(|X||Y|)^{1/2},
\end{align*}
and this proves the claim. Although this argument was quite straightforward, no significant improvement (in terms of lower bounds for $|X|$ and $|Y|$) is known, although it is widely believed to be true even for sets $X$ and $Y$ of sizes at least $\log ^C p$ for some constant $C$ -- a conjecture known in some literature \cite{flat} as the Paley graph conjecture.

Given that the set of quadratic residues is believed to have many properties in common with a random set of the same size, it is natural to ask whether the statement above is true if we replace $Q$ with a genuinely random set. In this paper we give the positive answer to this question, that is we prove the following theorem.

\begin{theorem}\label{T:mainadditive}
 Let $G$ be a group of size $N$ and $w\colon \mathbb{N}\to \mathbb{R}$ some function that tends to infinity. Let $A\subset G$ be a random subset obtained by putting every element of $G$ into $A$ independently with probability $\frac12$. Then the following holds with probability $1-o(1)$: for all sets $X,Y\subset G$, $|X|,|Y|\geq w(N)\log ^2 N$, the number of pairs $(x,y)\in X\times Y$ such that $xy\in A$ differs from $\frac12|X||Y|$ by $o\left(|X||Y|\right)$. The rate of convergence implied by the $o$-notation depends only on $w$.
\end{theorem}

If for a pair of subsets $X,Y\subset G$ we have that the number of pairs $(x,y)\in X\times Y$ such that $xy\in A$ differs from $\frac12|X||Y|$ by $\epsilon |X||Y|$, we will say that it is \emph{$\epsilon$-extracted} by the set $A$. If all the pairs of subsets of size as in the previous theorem are $\epsilon$-extracted by the set $A$, we will say that $A$ is an \emph{$\epsilon$-extractor}. Theorem \ref{T:mainadditive} shows, in this terminology, that a random subset of $G$ is $o(1)$-extractor with high probability. The reason for this terminology will be explained in Section \ref{S:furthercomments}.

A Fourier approach, as above but using Chernoff-type estimates for the Fourier coefficients instead of Gauss sum estimates, suffices to prove Theorem \ref{T:mainadditive} when, say, $|X|,|Y|\geq N^{0.51}$. Unfortunately, this argument doesn't work for sets of size smaller than $\sqrt{N}$.

In Section~\ref{S:mainargument} we will present a different argument that will enable us to prove Theorem~\ref{T:mainadditive}. A very important part of the argument is Proposition~\ref{P:crucial}, which we prove in Section~\ref{S:proofcrucial}. Section \ref{S:boundsforXY} is devoted for proving a bound on the sizes of $X$ and $Y$ for which Theorem \ref{T:mainadditive} doesn't hold. Namely, we prove that Theorem \ref{T:mainadditive} doesn't hold if we consider sets $X$ and $Y$ of size at least $C\log N\log\log N$, for arbitrary large constant $C>0$. Finally, in Section~\ref{S:furthercomments} we give some open problems left after this paper and explain the connection with randomness extractors.

Before moving on to the proof of the main theorem, let us also mention that one could ask and answer the analogous question in the Erd\H{o}s-R\'{e}nyi setting, that is one can easily prove the following folklore theorem.

\begin{theorem}\label{T:erdosrenyi}
 Let $G=(V,E)$ be a graph sampled as in the Erd\H{o}s-R\'{e}nyi model $G(N,\frac12)$, that is $G$ has a vertex set $V$ of size $N$ and contains each possible edge with probability $\frac12$, and these choices are all made independently. Let $w\colon \mathbb{N}\to \mathbb{R}$ be some function that tends to infinity. Then the following holds with probability $1-o(1)$: for all sets $X,Y\subset V$, $|X|,|Y|\geq w(N)\log  N$, $\frac12+o(1)$ of all the possible edges connecting an element of $X$ and an element of $Y$ are contained in $E$. The rate of convergence implied by the $o$-notation depends only on $w$.
\end{theorem}

\textsl{Acknowledgements.} I would like to thank Sean Eberhard, Ben Green, Jakub Konieczny and Freddie Manners for helpful discussions, and two anonymous reviewers for carefully reading the paper and giving valuable comments. I would also like to thank the Mathematical Institute, University of Oxford for funding my research.

\section{Notation}\label{S:notation}

Although most of the notation and conventions were implicitly introduced in the previous section, we include them here for the reader's convenience. The logarithm with base $2$ will be denoted by $\log_2$, and natural logarithm by $\log$. We will use the standard $O$-notation. To be concrete, for functions $f,g\colon \mathbb{N}\to \mathbb{R}$ we will write $f(n)=O(g(n))$ and $|f(n)|\lesssim g(n)$ if $|f(n)|\leq Cg(n)$ for some constant $C>0$. We will write $f(n)=o(g(n))$ if $f(n)/g(n)\to 0$ as $n$ tends to infinity. All the groups we will work with will be finite, and we equip each of them with the uniform probability measure. Of course, this reflects in our definitions of convolution, inner product and $\ell^2$-norm. We will also use $\mathbb{E}$ notation to denote the average over a set, thus
\begin{align*}
f\ast g(x) &= \mathbb{E}_{y\in G}f(y)g(y^{-1}x),\\
\langle f,g\rangle &= \mathbb{E}_{x\in G}f(x)g(x),\\
\|f\|_2 &= \left(\mathbb{E}_{x\in G}|f(x)|^2\right)^{1/2}.\\
\end{align*}

\section{The main argument}\label{S:mainargument}

We will use the notation introduced in Theorem \ref{T:mainadditive}. Define $\epsilon(N)=(10^5/w(N))^{1/6}$. We are going to prove that with probability only $o(1)$ there exist sets $X,Y\subset G$, both of size at least $w(N)\log ^2 N$ such that the proportion of pairs $(x,y)\in X\times Y$ satisfying $xy\in A$ differs from $\frac12$ by at least $\epsilon(N)$. In order to make the random variables that will appear later centred, we associate to the set $A$ the function $f_A(x)=2\cdot 1_A-1$. Note that $f_A$ takes values in $\{-1,1\}$. From this perspective, it is easily seen that we are actually interested in whether
$$\left|\left\langle \frac{N^2\cdot1_X\ast 1_Y}{|X||Y|}, f_A\right\rangle\right| \geq 2\epsilon(N).$$
We can try to bound this using the following classical Chernoff-type concentration result by Hoeffding \cite{hoeffding}.

\begin{proposition}[Hoeffding]\label{P:concentration}
Let $X_1,\dots, X_n$ be independent bounded random variables with mean $0$. Then for every $\lambda\geq 0$ we have
$$\mathbb{P}\left(\left|\sum_iX_i\right|\geq \lambda\sqrt{\sum_i\|X_i\|_{\infty}^2}\right)\lesssim \exp(-\lambda^2/2).$$
\end{proposition}

Applying this to our problem we see that for fixed $X$ and $Y$ we have
$$\mathbb{P}\left(\left|\left\langle \frac{N^2\cdot 1_X\ast 1_Y}{|X||Y|}, f_A\right\rangle\right|\geq 2\epsilon(N)\right)\lesssim \exp\left(-\frac{2\epsilon(N)^2|X|^2|Y|^2}{N^3\|1_X\ast 1_Y\|_2^2}\right).$$
This looks promising but is still too weak for the union bound (over all sets of fixed sizes) to be effective; concretely, it is easy to see that the problem is caused by pairs $(X,Y)$ for which $\|1_X\ast 1_Y\|_2^2$ is close to maximal possible (that is $|X||Y|\min(|X|, |Y|)$).

Our approach is the following. We will prove that whenever there are pairs of sets $X$ and $Y$ with an unusual proportion of elements whose product is in $A$, we can find some related sets $S$ and $T$ that share this property, but, crucially, such that $\|1_S\ast 1_T\|_2^2$ is significantly smaller than the maximum possible so the final union bound would be effective. This rough idea of relating arbitrary sets to unstructured ones and then continue working with these is already present in some previous works and it goes back at least as far as \cite{visibility}.

The following proposition gives us sets $S$ and $T$ with these nice properties. 

\begin{proposition}\label{P:crucial}
Let $f\colon G\to\{-1,1\}$ be any function, and $X,Y\subset G$ both of size at least $w(N)\log ^2 N$ and $|X|\leq |Y|$. There exist sets $S,T\subset G$ satisfying the following properties
\begin{enumerate}
\item $|S|\leq |T|$ and $|T|\to\infty$,
\item $\displaystyle\left|\left\langle\frac{N^2\cdot 1_S\ast 1_T}{|S||T|},f\right\rangle - \left\langle\frac{N^2\cdot 1_X\ast 1_Y}{|X||Y|},f\right\rangle\right| \leq \epsilon(N),$
\item $\displaystyle\frac{|S|^2|T|^2}{N^3\|1_S\ast 1_T\|_2^2} \geq \frac{6|T|\log  N}{\epsilon(N)^2}.$
\end{enumerate}
\end{proposition}

We give a proof of this proposition in the next section, and here finish the proof of the main theorem conditional on it. 

Suppose that there are sets $X,Y$ of size at least $w(N)\log ^2 N$ such that $\left|\left\langle \frac{N^2\cdot 1_X\ast 1_Y}{|X||Y|}, f_A\right\rangle\right| \geq 2\epsilon(N)$. Take subsets $S$ and $T$ given by Proposition~\ref{P:crucial}; by the triangle inequality we have $\left|\left\langle \frac{N^2\cdot 1_S\ast 1_T}{|S||T|}, f_A\right\rangle\right| \geq \epsilon(N)$. However, by our choice of $S$ and $T$ this is very unlikely to happen; applying Hoeffding's inequality as above gives that the probability is bounded above by
$$\exp\left( -\frac{\epsilon(N)^2|S|^2|T|^2}{2N^3\|1_S\ast 1_T\|_2^2}\right) \leq \exp\left( -3|T|\log  N\right).$$
Now we do the union bound over all sets $S$ and $T$ of fixed sizes satisfying the properties stated in Proposition~\ref{P:crucial}. We use a crude bound $N^{2|T|}$ for the number of these sets (because $|S|\leq |T|$). The contribution coming from these sets is hence bounded by
$$N^{2|T|} e^{-3|T|\log  N} = e^{-|T|\log  N}.$$
What remains is to sum over all possible sizes of $S$ and $T$. There are $N^2$ terms and since the size of $T$ tends to infinity, the total sum is still $o(1)$.

\section{Proof of the main proposition}\label{S:proofcrucial}

In this section we give the probabilistic proof of Proposition~\ref{P:crucial}. The idea is to take a random subset $S\subset X$ uniformly from all subsets of some fixed size $s$, and a random subset $T\subset Y$ uniformly from all subsets of some fixed size $t$. Intuitively, we believe that $1_S\ast 1_T$ should in some way imitate $1_X\ast 1_Y$. This idea is already present in the work of Croot and Sisask \cite{crootsisask}. Additionally, if $s$ and $t$ are somewhat smaller than $|X|$ and $|Y|$, $\|1_S\ast 1_T\|_2^2$ should be smaller than the theoretical maximum, even if $X$ and $Y$ are highly structured. 

We formalise the outlined strategy. In the following, $S$ and $T$ are sets chosen randomly as described, and hence instances of $\mathbb{P}$ (probability) and $\mathbb{E}$ (expectation) notation are over this random choice, unless indicated otherwise. We first prove that the additive energy is on average small.

\begin{lemma}\label{L:additive}
$$\mathbb{E}\|1_S\ast 1_T\|_2^2 \leq \frac{st}{N^3} + \frac{s^2t^2}{|X|^2|Y|^2}\|1_X\ast 1_Y\|_2^2.$$
\end{lemma}
\begin{proof}
Note that the left hand side is equal to
$$\frac{st}{N^3} + \frac{1}{N^3}\sum \mathbb{P}(x_1,x_2\in S, y_1,y_2\in T),$$
where the sum is taken over $x_1,x_2\in X$, $y_1,y_2\in Y$ satisfying $x_1\neq x_2$, $y_1\neq y_2$, $x_1y_1=x_2y_2$. This is obviously bounded by
$$\frac{st}{N^3} + \|1_X\ast 1_Y\|_2^2\cdot \frac{s(s-1)}{|X|(|X|-1)}\frac{t(t-1)}{|Y|(|Y|-1)},$$
and hence we are done.
\end{proof}

\begin{lemma}\label{L:closenessSTXY}
For any function $f\colon G\to\{-1,1\}$ we have
 $$\mathbb{E}\left|\left\langle\frac{N^2\cdot 1_S\ast 1_T}{|S||T|},f\right\rangle - \left\langle\frac{N^2\cdot 1_X\ast 1_Y}{|X||Y|},f\right\rangle\right| \leq 2\sqrt{\frac{|Y|}{st}}$$
\end{lemma}

\begin{proof}
Notice that
\begin{align*}
 &\mathbb{E}\left|\left\langle\frac{N^2\cdot 1_S\ast 1_T}{|S||T|},f\right\rangle - \left\langle\frac{N^2\cdot 1_X\ast 1_Y}{|X||Y|},f\right\rangle\right| = \mathbb{E} \left| \sum_{z\in G}\sum_{w\in G}f(z)\cdot\left(\frac{1_S(w)1_T(w^{-1}z)}{st} - \frac{1_X(w)1_Y(w^{-1}z)}{|X||Y|}\right)\right|\\
 &= \mathbb{E} \left| \sum_{x\in X}\sum_{y\in Y}f(xy)\cdot\left(\frac{1_S(x)1_T(y)}{st} - \frac{1}{|X||Y|}\right)\right|\\
 &= \frac{1}{st}\mathbb{E} \left| \sum_{x\in X}\sum_{y\in Y}f(xy)\cdot\left(1_S(x)1_T(y) - \mathbb{P}(x\in S, y\in T)\right)\right|\\
\end{align*}
The idea now is to exploit the fact that lots of the random variables in the sum above are independent, and hence we expect a fair amount of cancellation. To do this we will partition the set $X\times Y$ into few parts, such that corresponding random variables in each part are (very close to) independent. This partition is easy to describe if we identify $X$ and $Y$ with $\{1,\dots,|X|\}$ and $\{1,\dots,|Y|\}$; for $i=1,\dots, |Y|$, let $\Omega_i=\{(j,j+i)\colon j\in \{1,\dots,|X|\}\}$ (with addition modulo $|Y|$). We first analyse the sum above when taken over all pairs with $f(xy)=1$; we will denote this set by $\mathcal{F}$. We have
\begin{align}
 &\frac{1}{st}\mathbb{E} \left| \sum_{(x,y)\in \mathcal{F}}\left(1_S(x)1_T(y) - \mathbb{P}(x\in S, y\in T)\right)\right| \leq \frac{1}{st}\sum_{i=1}^{|Y|}\mathbb{E}\left|\sum_{(x,y)\in \Omega_i\cap\mathcal{F}}\left(1_S(x)1_T(y) - \mathbb{P}(x\in S, y\in T)\right)\right|\nonumber\\
 &\leq \frac{1}{st}\sum_{i=1}^{|Y|}\left(\mathbb{E}\left(\sum_{(x,y)\in \Omega_i\cap\mathcal{F}}\left(1_S(x)1_T(y) - \mathbb{P}(x\in S, y\in T)\right)\right)^2\right)^{1/2}\nonumber\\
 &= \frac{1}{st}\sum_{i=1}^{|Y|} \left(\sum_{(x_j,y_j)\in \Omega_i\cap\mathcal{F}}\left(\mathbb{P}(x_1,x_2\in S, y_1,y_2\in T) - \mathbb{P}(x_1\in S, y_1\in T)\mathbb{P}(x_2\in S, y_2\in T)\right)\right)^{1/2}.\label{E:zadnjaAlijepo}
\end{align}
Consider two elements $(x_1,y_1)\neq (x_2,y_2)$ from $\Omega_i\cap\mathcal{F}$. Notice that, by the definition of $\Omega_i$, this implies that $x_1\neq x_2$ and $y_1\neq y_2$. We have
$$\mathbb{P}(x_1,x_2\in S, y_1,y_2\in T) = \frac{s(s-1)}{|X|(|X|-1)}\frac{t(t-1)}{|Y|(|Y|-1)}$$
and
$$\mathbb{P}(x_1\in S, y_1\in T)\mathbb{P}(x_2\in S, y_2\in T) = \frac{s^2t^2}{|X|^2|Y|^2},$$
so the contribution to \eqref{E:zadnjaAlijepo} coming from nondiagonal terms is nonpositive. Hence
$$\eqref{E:zadnjaAlijepo}\leq \frac{1}{st}\sum_{i=1}^{|Y|}\left(\sum_{(x,y)\in \Omega_i\cap\mathcal{F}}\mathbb{P}(x\in S, y\in T)\right)^{1/2} \leq \frac{1}{\sqrt{st|X||Y|}}\sum_{i=1}^{|Y|}|\Omega_i|^{1/2}\leq \sqrt{\frac{|Y|}{st}}.$$
Doing the same thing for the other sum (over $\mathcal{F}^{c}$), we get the statement of the lemma.
\end{proof}

We can now finish the proof of Proposition~\ref{P:crucial} using Markov's inequality. Let $s=\frac{2000\log N}{\epsilon(N)^4}$ and $t=\frac{|Y|\epsilon(N)^2}{50\log N}$. By Lemma~\ref{L:closenessSTXY} we see that the probability that
\begin{equation}\label{e:cond1}
 \left|\left\langle\frac{N^2\cdot 1_S\ast 1_T}{st},f\right\rangle - \left\langle\frac{N^2\cdot1_X\ast 1_Y}{|X||Y|},f\right\rangle\right| \leq 6\sqrt{\frac{|Y|}{st}}\leq \epsilon(N)
\end{equation}
is at least $\frac23$. Similarly, Lemma~\ref{L:additive} implies that
\begin{equation}\label{e:cond2}
 \|1_S\ast 1_T\|_2^2 \leq \frac{3st}{N^3} + \frac{3s^2t^2}{|X|^2|Y|^2}\|1_X\ast 1_Y\|_2^2,
\end{equation}
happens with probability at least $\frac23$ so with positive probability we have both of the requirements satisfied. Inequality \eqref{e:cond2} implies
$$\frac{s^2t^2}{N^3\|1_S\ast 1_T\|_2^2}\geq \min\left(st/6, \frac{|X|^2|Y|^2}{6N^3\|1_X\ast 1_Y\|_2^2}\right)\geq \min(st/{6}, |Y|/6)\geq \frac{6t\log N}{\epsilon(N)^2},$$
and this finishes the proof.

\section{Bounds for the sizes of extracted sets}\label{S:boundsforXY}

Recall that Theorem \ref{T:mainadditive} states that with high probability all pairs of subsets $(X,Y)$ of size at least $w(N)\log^2 N$ are $o(1)$-extracted by a random subset. In this section we deal with the question of how far the bound $w(N)\log^2 N$ is from optimal. We note that some of the calculations here have already been done in \cite[Section 8]{green}. Indeed, the main argument there already gives a nontrivial lower bound on the minimum size we must impose on $X$ and $Y$ in order for Theorem \ref{T:mainadditive} to hold. It is an easy consequence of Green's argument that with probability $\frac12-o(1)$ there is a subset $X\subset \mathbb{F}_2^n$, $N=2^n$ of size $\frac14\log_2 N\log_2\log_2 N$ such that $X+X\subset A$. Obviously, $X+X\subset A$ is a much stronger statement than is required to contradict the conclusion of Theorem~\ref{T:mainadditive}. In this section we show that the constant $\frac14$ can be made arbitrarily large if we relax this statement so that Theorem~\ref{T:mainadditive} only just fails. The main theorem is as follows.

\begin{theorem}\label{T:bounds}
	Fix $C>0$ and let $A$ be a random subset of $\mathbb{F}_2^n$. There exists $\epsilon > 0$ such that with probability bounded away from $0$ there exists a subset $X\subset \mathbb{F}_2^n$ of size at least $C\log N\log\log N$ such that for at least $\frac12+\epsilon$ of the pairs $(x_1,x_2)\in X\times X$ we have $x_1+x_2\in A$.
\end{theorem}

As will be obvious from the proof, we will actually choose $X$ to be a subspace of relatively large dimension. Notice that in this case we have that
$$|\{(x_1,x_2)\colon x_1,x_2\in X, x_1+x_2\in A\}| = |X||A\cap X|,$$
so what we need to prove is that $|A\cap X|\geq (\frac12+\epsilon)|X|$. We should also mention that in this setting (i.e.\ when $X$ must be a subspace), the bound in Theorem \ref{T:bounds} is optimal up to a constant. This is easily proven using the first moment method.

Before moving to the proof, we state two simple counting lemmas we will need. As both are quite standard, we omit the proofs which can be found in e.g.\ \cite{green}.

\begin{lemma}\label{L:Ik}
	Let $I_k$ be the number of $k$-dimensional subspaces of $\mathbb{F}_2^n$. Then $I_k\geq 2^{nk-k^2}$.
\end{lemma}

\begin{lemma}\label{L:Jkl}
	Let $J_k^{(l)}$ be the number of pairs of $k$-dimensional subspaces of $\mathbb{F}_2^n$ whose intersection is a $l$-dimensional subspace. Then $J_k^{(l)}\leq 2^{2nk-nl}$.
\end{lemma}

In the course of the proof of Theorem \ref{T:bounds} we will also need rather precise upper bounds for an upper tail of a binomial distribution, more precise than the one provided by Proposition \ref{P:concentration}. Additionally, we will need a corresponding lower bound. Both of these are provided by the following proposition.

\begin{proposition}\label{P:boundsbinomial}
Let $Z$ be a $B(n,\frac12)$ random variable and $0<\epsilon<\frac12$ fixed. If we define $h(x) = (\frac12+x)\log (1+2x) + (\frac12-x)\log (1-2x)$, then
$$e^{-nh(\epsilon)-O(\log n)} \leq \mathbb{P}(Z\geq (\textstyle{\frac12}+\epsilon) n) \leq e^{-nh(\epsilon)}.$$
\end{proposition}

This proposition is relatively standard and can be found in \cite[Appendix A]{alonspencer}.

We finally move to the proof of the main result in this section. As we noted above, we will use the second moment method. Fix $\epsilon > 0$ and let $k = \lfloor \log_2 n + \log_2\log_2 n\rfloor + c$, where $c>0$ is a constant to be specified later. Let $\{V_i\}$ be the family of all $k$-dimensional subspaces of $\mathbb{F}_2^n$ and let $X_i$ be the random variable indicating if $|A\cap V_i|\geq (\frac12+\epsilon)|V_i|$ and $X=\sum_i X_i$ be the number of $V_i$ such that $|A\cap V_i|\geq (\frac12+\epsilon)|V_i|$. We claim that with probability bounded away from zero we have $X>0$, equivalently, with probability bounded away from one we have $X=0$.

Let $P_k=\mathbb{P}(|A\cap V_i|\geq (\frac12+\epsilon)|V_i|)$ and $C_k^{(l)}=\mathrm{cov}(X_i,X_j)$ for $i\neq j$ and $\mathrm{dim} (V_i\cap V_j) = l$. We will use the notation introduced in the statements of Lemma \ref{L:Ik}, Lemma \ref{L:Jkl} and Proposition \ref{P:boundsbinomial}.

Obviously, $\mathbb{E}X = I_kP_k$. On the other hand
$$\mathrm{Var} X = \sum_{i,j}\mathrm{cov}(X_i,X_j) = \sum_{l=0}^k J_k^{(l)}C_k^{(l)}.$$

Note that by Chebyshev's inequality we have
\begin{equation}\label{E:chebyshev}
	\mathbb{P}(X=0)\leq \frac{\mathrm{Var} X}{(\mathbb{E}X)^2} = \sum_{l=0}^k I_k^{-2}P_k^{-2}J_k^{(l)}C_k^{(l)},
\end{equation}
and our aim now is to prove that the quantity on the right is bounded away from $1$.

We will inspect three different ranges for $l$. To begin with, let $l=0$. Obviously, $I_k^{-2}J_k^{(0)}\leq 1$. If we let $K=2^k$, $H = (\frac12+\epsilon)K$ and $Z$ be a $B(K-1,\frac12)$ random variable, by conditioning on whether $0$ (which is the only element of $V\cap W$) is in $A$, it is easy to see that
\begin{align*}
C_k^{(0)} &= \textstyle{\frac12}\mathbb{P}(Z\geq H-1)^2 + \textstyle{\frac12}\mathbb{P}(Z\geq H)^2 - \left(\textstyle{\frac12}\mathbb{P}(Z\geq H-1) + \textstyle{\frac12}\mathbb{P}(Z\geq H)\right)^2\\
&= \textstyle{\frac14}\left(\mathbb{P}(Z\geq H-1) - \mathbb{P}(Z\geq H)\right)^2.
\end{align*}
On the other hand, by the same conditioning we get
$$P_k = \textstyle{\frac12}\mathbb{P}(Z\geq H-1) + \textstyle{\frac12}\mathbb{P}(Z\geq H),$$
and hence
$$P_k^{-2}C_k^{(0)}\leq \left(\frac{\mathbb{P}(Z= H-1)}{\mathbb{P}(Z= H-1) + \mathbb{P}(Z= H)}\right)^2 \leq (\textstyle{\frac12}+\epsilon)^2,$$
where the last inequality follows by comparing appropriate binomial coefficients. Combining these observations we get
\begin{equation}\label{E:l0}
I_k^{-2}P_k^{-2}J_k^{(0)}C_k^{(0)}\leq (\textstyle{\frac12}+\epsilon)^2.
\end{equation}

Suppose now that $1\leq l\leq \log_2 n$, and let $d = k-l$ and $D=2^d$. Obviously, $D\geq \log_2 n$. We will bound the $C_k^{(l)}$ by the probability 
$$\mathbb{P}(|A\cap V_i|\geq (\textstyle{\frac12}+\epsilon)|V_i|, |A\cap V_j|\geq (\textstyle{\frac12}+\epsilon)|V_j|),$$
where $i\neq j$ and $V_i$ and $V_j$ have $l$-dimensional intersection $W$. Notice that $|V_i\setminus W|= |V_j\setminus W| = \textstyle{\frac{D-1}{D}}2^k$, and if we want both of the events to hold, then certainly $V_1\setminus W$ and $V_2\setminus W$ each must contain at least $(\frac12+\epsilon-1/D)2^k$ elements from $A$, since $W$ has $2^k/D$ elements. However, by Proposition \ref{P:boundsbinomial}, the probability of this happening is at most
\begin{align*}
\exp\left(-2\cdot \textstyle{\frac{D-1}{D}}2^kh(\frac{\epsilon D}{D-1}-\frac{1}{2(D-1)})\right) &= \exp\left(-2^{k+1}h(\textstyle{\frac{\epsilon D}{D-1}}-\textstyle{\frac{1}{2(D-1)}}) + O(2^k / D)\right)\\
&= \exp\left(-2^{k+1}h(\epsilon) + O(2^k / D)\right)\\
&= \exp\left(-2^{k+1}h(\epsilon) + O(2^l)\right).
\end{align*}
The second equality above follows by the intermediate value theorem. After using Lemma \ref{L:Jkl} and noticing that for each $1\leq l\leq \log_2 n$ we have $O(2^l)-nl\leq O(1)-n$ we can conclude
$$\sum_{l=1}^{\log_2 n}J_k^{(l)}C_{k,l} \leq 2^{2nk - n + O(\log_2\log_2 n)}e^{-2^{k+1}h(\epsilon)}.$$
We can now use the lower bound from Proposition \ref{P:boundsbinomial} to bound $P_k$
$$P_k\geq e^{-2^kh(\epsilon)-O(k)}.$$
Using this and Lemma \ref{L:Ik} we get
\begin{equation}\label{E:l1logn}
\sum_{l=1}^{\log_2 n}I_k^{-2}P_k^{-2}J_k^{(l)}C_k^{(l)} \leq 2^{2k^2 + O(k) - n},
\end{equation}
and this is $o(1)$.

Finally, for $l\geq \log_2 n$ we can bound $C_k^{(l)}$ trivially by $1$, and get, similarly as in the previous case,
\begin{equation}\label{E:lloglogn}
\sum_{l=\log_2 n}^kI_k^{-2}P_k^{-2}J_k^{(l)}C_k^{(l)} \leq k2^{2k^2 - n\log_2 n} e^{2^{k+1}h(\epsilon) + O(k)} \leq 2^{2k^2 -n\log_2 n} e^{2^{c+1}h(\epsilon)n\log_2 n + O(k)}.
\end{equation}
This is $o(1)$ as long as $2^{c+1}h(\epsilon) < \log 2$. Now from \eqref{E:chebyshev}, \eqref{E:l0}, \eqref{E:l1logn} and \eqref{E:lloglogn}
$$P(X=0)\leq (\textstyle{\frac12}+\epsilon)^2+o(1),$$
as long as $2^{c+1}h(\epsilon) < \log 2$. Since $h(\epsilon)\to 0$ as $\epsilon\to 0$, we see that we can have $c$ arbitrary large if we choose $\epsilon$ small enough.

\section{Further comments}\label{S:furthercomments}

We have phrased Theorem \ref{T:mainadditive} in terms of the general finite group $G$. One could, of course, ask the same question for a specific class of groups, say, $\mathbb{Z}/N\mathbb{Z}$ or $\mathbb{F}_2^n$. For example, it is possible to prove that in $\mathbb{Z}/N\mathbb{Z}$ with high probability there exists a subset $X$ of size about $2\log_2 N$ such that $X+X\subset A$. This already gives a bound for the sizes of sets which we can hope to be extracted. One could presumably push this a bit further similarly as we did for $\mathbb{F}_2^n$ in Theorem~\ref{T:bounds}. 

Notice that in the Erd\H{o}s-R\'{e}nyi model, the bound $w(N)\log_2 N$ in Theorem~\ref{T:erdosrenyi} is quite close to optimal. Indeed, the clique number in this graph is with high probability about $2\log_2 N$. One could consider this as a (weak) indication that it might be possible to push the bound in Theorem~\ref{T:mainadditive} further down, maybe even to something like $w(N)\log N\log\log N$. However, methods such as those in this paper are not sufficient to achieve anything close to that.

Our result can also be considered from the perspective of randomness extractors. We say that a function is a randomness extractor if it transforms every (e.g.\ pair or sequence of) not too degenerate independent random variables into a Bernoulli random variable which is quite close to uniform. In other words, a randomness extractor takes a bad nonuniform source and creates a nice almost uniform output.

A folklore example of this phenomena is the so-called von Neumann extractor. Suppose that $(X_n)$ is a sequence of independent identically distributed nondegenerate Bernoulli random variables. One would like to construct a function of this sequence which produces a symmetric Bernoulli random variable. Let $f((x_n))=x_m$, where $m$ is the smallest even index such that $x_m\neq x_{m+1}$. It is easy to see that this function does the job.

Of course, the definition as stated above is far from precise, we give a rigorous definition adapted to the setting we will be interested in. Given a random variable $X$ with values in a finite set $S$, we measure how nondegenerate it is (or, at the other extreme, how uniform it is) with normalized min-entropy (we will call it entropy from now on)
$$H(X)=\frac{\min_{s\in S}\log 1/\mathbb{P}(X=s)}{\log |S|}.$$
To gain some intuition, it is worth mentioning that entropy takes values between $0$ and $1$, and that $H(X)=0$ if and only if $X$ concentrates on a single value (i.e.\ is completely deterministic), and $H(X)=1$ if and only if $X$ is uniform on $S$ (i.e.\ is as unpredictable as one could possibly hope).

Fix $\epsilon, c>0$ (one should think of $\epsilon$ as being close to $0$ and $c$ as close to $1$). We will say that $f\colon S\times S\to\{0,1\}$ is a $(\epsilon, c)$-randomness extractor if for every pair of independent random variables $X$ and $Y$ with values on $S$ satisfying $H(X),H(Y) > \epsilon$, we have $H(f(X,Y))>c$. 

It was proven by Chor and Goldreich \cite{flat}, that it is only necessary to check that $f(X,Y)$ extracts randomness well (i.e.\ satisfies the required lower bound on the entropy) when $X$ and $Y$ are \emph{flat}, that is to say distribution of each of them is supported and uniformly distributed on a subset of $S$. Having this in mind, our result proves that for any fixed $\delta>0$, for a randomly chosen set $A$ of density $\frac12$, the function $(x,y)\to 1_A(xy)$ is $(\delta,1-o(1))$-randomness extractor with probability $1-o(1)$. It is a major open problem in the field to find explicit $(\delta,1-o(1))$-randomness extractors for small values of $\delta$. For example, to the best of the author's knowledge, the only explicit $(\delta,1-o(1))$-randomness extractor for which $\delta < \frac12$ was constructed by Bourgain \cite{bourgain}. For more about this and related issues see e.g.\ the first section of the paper by Barak, Impagliazzo, and Wigderson \cite{survey}.

\bibliography{extractors}{}

\begin{thebibliography}{AAAS94}

\bibitem[AAAS94]{visibility}
P.~K. Agarwal, N.~Alon, B.~Aronov, and S.~Suri.
\newblock Can visibility graphs be represented compactly?
\newblock {\em Discrete Comput. Geom.}, 12(3):347--365, 1994.
\newblock ACM Symposium on Computational Geometry (San Diego, CA, 1993).

\bibitem[AS08]{alonspencer}
Noga Alon and Joel~H. Spencer.
\newblock {\em The probabilistic method}.
\newblock Wiley-Interscience Series in Discrete Mathematics and Optimization.
  John Wiley \& Sons, Inc., Hoboken, NJ, third edition, 2008.
\newblock With an appendix on the life and work of Paul Erd{\H{o}}s.

\bibitem[BIW06]{survey}
Boaz Barak, Russell Impagliazzo, and Avi Wigderson.
\newblock Extracting randomness using few independent sources.
\newblock {\em SIAM J. Comput.}, 36(4):1095--1118 (electronic), 2006.

\bibitem[Bou05]{bourgain}
Jean Bourgain.
\newblock More on the sum-product phenomenon in prime fields and its
  applications.
\newblock {\em Int. J. Number Theory}, 1(1):1--32, 2005.

\bibitem[CG88]{flat}
Benny Chor and Oded Goldreich.
\newblock Unbiased bits from sources of weak randomness and probabilistic
  communication complexity.
\newblock {\em SIAM J. Comput.}, 17(2):230--261, 1988.
\newblock Special issue on cryptography.

\bibitem[CS10]{crootsisask}
Ernie Croot and Olof Sisask.
\newblock A probabilistic technique for finding almost-periods of convolutions.
\newblock {\em Geom. Funct. Anal.}, 20(6):1367--1396, 2010.

\bibitem[Gre05]{green}
Ben Green.
\newblock Counting sets with small sumset, and the clique number of random
  {C}ayley graphs.
\newblock {\em Combinatorica}, 25(3):307--326, 2005.

\bibitem[Hoe63]{hoeffding}
Wassily Hoeffding.
\newblock Probability inequalities for sums of bounded random variables.
\newblock {\em J. Amer. Statist. Assoc.}, 58:13--30, 1963.

\bibitem[IK04]{kowalskiiwaniec}
Henryk Iwaniec and Emmanuel Kowalski.
\newblock {\em Analytic number theory}, volume~53 of {\em American Mathematical
  Society Colloquium Publications}.
\newblock American Mathematical Society, Providence, RI, 2004.

\end{thebibliography}
\bibliographystyle{alpha}

\end{document}